\documentclass{article}

\usepackage[margin=1.0in]{geometry}
\usepackage{amsmath}
\usepackage{amssymb}
\usepackage{amsthm}
\usepackage{enumerate}
\newtheorem{lem}{Lemma}
\newtheorem{thm}{Theorem}
\newtheorem{dfn}{Definition}

\newtheorem*{prb*}{Problem}
\newtheorem*{thm*}{Theorem}
\providecommand{\e}{\varepsilon}

\providecommand{\ph}{\varphi}

\title{Areas of triangles and $\text{SL}_2$ actions in finite rings}
\author{Alex McDonald}

\begin{document}

\maketitle

\begin{abstract}
In Euclidean space, one can use the dot product to give a formula for the area of a triangle in terms of the coordinates of each vertex.  Since this formula involves only addition, subtraction, and multiplication, it can be used as a definition of area in $R^2$, where $R$ is an arbitrary ring.  The result is a quantity associated with triples of points which is still invariant under the action of $\text{SL}_2(R)$.  One can then look at a configuration of points in $R^2$ in terms of the triangles determined by pairs of points and the origin, considering two such configurations to be of the same type if corresponding pairs of points determine the same areas.  In this paper we consider the cases $R=\mathbb{F}_q$ and $R=\mathbb{Z}/p^\ell \mathbb{Z}$, and prove that sufficiently large subsets of $R^2$ must produce a positive proportion of all such types of configurations.
\end{abstract}

\section{Introduction}

There are several interesting combinatorial problems asking whether a sufficiently large subset of a vector space over a finite field must generate many different objects of some type.  The most well known example is the Erdos-Falconer problem, which asks whether such a set must contain all possible distances, or at least a positive proportion of distances.  More precisely, given $E\subset\mathbb{F}_q^d$ we define the distance set

\[
\Delta(E)=\{(x_1-y_1)^2+\cdots +(x_d-y_d)^2:x,y\in E\}.
\]

Obviously, $\Delta(E)\subset\mathbb{F}_q$.  The Erdos-Falconer problem asks for an exponent $s$ such that $\Delta(E)=\mathbb{F}_q$, or more generally $|\Delta(E)|\gtrsim q$, whenever $|E|\gtrsim q^s$ (Throughout, the notation $X\lesssim Y$ means there is a constant $C$ such that $X\leq CY$, $X\approx Y$ means $X\lesssim Y$ and $Y\lesssim X$, and $O(X)$ denotes a quantity that is $\lesssim X$).  In \cite{IR}, Iosevich and Rudnev proved that $\Delta(E)=\mathbb{F}_q$ if $|E|\gtrsim q^{\frac{d+1}{2}}$   In \cite{Sharpness} it is proved by Hart, Iosevich, Koh, and Rudnev that the exponent $\frac{d+1}{2}$ cannot be improved in odd dimensions, althought it has been improved to $4/3$ in the $d=2$ case (first in \cite{WolffExponent} in the case $q\equiv 3 \text{ (mod 4)}$ by Chapman, Erdogan, Hart, Iosevich, and Koh, then in general in \cite{GroupAction} by Bennett, Hart, Iosevich, Pakianathan, and Rudnev).   Several interesting variants of the distance problem have been studied as well.  A result of Pham, Phuong, Sang, Valculescu, and Vinh studies the problem when distances between pairs of points are replaced with distances between points and lines in $\mathbb{F}_q^2$; they prove that if sets $P$ and $L$ of points and lines, respectively, satisfy $|P||L|\gtrsim q^{8/3}$, then they determine a positive proportion of all distances \cite{Lines}.  Birklbauer, Iosevich, and Pham proved an analogous result about distances determined by points and hyperplanes in $\mathbb{F}_q^d$ \cite{Planes}.\\

We can replace distances with dot products and ask the analogous question.  Let

\[
\Pi(E)=\{x_1y_1+\cdots +x_dy_d:x,y\in E\},
\]

and again ask for an exponent $s$ such that $|E|\gtrsim q^s$ implies $\Pi(E)$ contains all distances (or at least a positive proportion of distances).  Hart and Iosevich prove in \cite{HI} that the exponent $s=\frac{d+1}{2}$ works for this question as well.  The proof is quite similar to the proof of the same exponent in the Erdos-Falconer problem; in each case, the authors consider a function which counts, for each $t\in\mathbb{F}_q$, the number of representations of $t$ as, respectively, a distance and a dot product determined by the set $E$.  These representation functions are then studied using techniques from Fourier analysis.  \\

Another interesting variant of this problem was studied in \cite{Angles}, where Lund, Pham, and Vinh defined the angle between two vectors in analogue with the usual geometric interpretation of the dot product.  Namely, given vectors $x$ and $y$, they consider the quantity

\[
s(x,y)=1-\frac{(x\cdot y)^2}{\|x\|\|y\|},
\]

where $\|x\|=x_1^2+\cdots x_d^2$ is the finite field distance defined above.  Note that since we cannot always take square roots in finite fields, the finite field distance corresponds to the square of the Euclidean distance; therefore, $s(x,y)$ above is the correct finite field analogue of $\sin^2\theta$, where $\theta$ is the angle between the vectors $x$ and $y$.  This creates a variant of the dot product problem, since one can obtain different dot products from the same angle by varying length.  The authors go on to prove that the exponent $\frac{d+2}{2}$ guarantees a positive proportion of angles.  \\

It is of interest to generalize these types of results to point configurations.  By a $(k+1)$-point configuration in $\mathbb{F}_q^d$, we simply mean an element of $(\mathbb{F}_q^d)^{k+1}$.  Throughout, we will use superscripts to denote different vectors in a given configuration, and subscripts to denote the coordinates of each vector.  For example, a $(k+1)$ point configuration $x$ is made up of vectors $x^1,...,x^{k+1}$, each of which has coordinates $x_1^i,x_2^i$.  Given a set $E\subset\mathbb{F}_q^d$, we can consider $(k+1)$-point configurations in $E$ (i.e., elements of $E^{k+1}$) and ask whether $E$ must contain a positive proportion of all configurations, up to some notion of equivalence.  For example, we may view $(k+1)$-point configurations as simplices, and our notion of equivalence is geometric congruence; any two simplices are congruent if there is a translation and a rotation that maps one onto the other.  Since a $2$-simplex is simply a pair of points, and two such simplices are congruent if and only if the distance is the same, congruence classes simply correspond to distance.  Hence, the Erdos-Falconer distance problem may be viewed as simply the $k=1$ case of the simplex congruence problem.  In \cite{Ubiquity}, Hart and Iosevich prove that $E$ contains the vertices of a congruent copy of every non-degenerate simplex (non-degenerate here means the points are in general position) whenever $|E|\gtrsim q^{\frac{kd}{k+1}+\frac{k}{2}}$.  However, in order for this result to be non-trivial the exponent must be $<d$, and that only happens when $\binom{k+1}{2}<d$.  So, the result is limited to fairly small configurations.  This result is improved in \cite{GroupAction} by Bennett, Hart, Iosevich, Pakianathan, and Rudnev, who prove that for any $k\leq d$ a set $E\subset\mathbb{F}_q^d$ determines a positive proportion of all congruence classes of $(k+1)$-point configurations provided $|E|\gtrsim q^{d-\frac{d-1}{k+1}}$.  This exponent is clearly non-trivial for all $k$.  In \cite{Me}, I extended this result to the case $k\geq d$.  \\

In this paper, we consider a different notion of equivalence.  We will consider the problem over both finite fields and rings of integers modulo powers of primes, so I will define the equivalence relation in an arbitrary ring.

\begin{dfn}
Let $R$ be a ring, and let $E\subset R^2$.  We define an equivalence relation $\sim$ on $E^{k+1}$ by $(x^1,...,x^{k+1})\sim (y^1,...,y^{k+1})$ (or more breifly $x\sim y$) if and only if for each pair $i,j$ we have $x^i\cdot x^{j\perp}=y^i\cdot y^{j\perp}$.  Define $\mathcal{C}_{k+1}(E)$ to be the set of equivalence classes of $E$ under this relation.
\end{dfn}

In the Euclidean setting, $\frac{1}{2} |x\cdot y^\perp|$ is the area of the triangle with vertices $0,x,y$.  So, we may view each pair of points in a $(k+1)$-point configuration as determining a triangle with the origin, and we consider two such configurations to be equivalent if the triangles they determine all have the same areas.  As we will prove in section $2$, this equivalence relation is closely related to the action of $\text{SL}_2(R)$ on tuples of points; except for some degenerate cases, two configurations are equivalent if and only if there is a unique $g$ mapping one to the other.  This allows us to analyze the problem in terms of this action; in section 2, we define a counting function $f(g)$ and reduce matters to estimating the sum $\sum_g f(g)^{k+1}$.  In section 3, we show how to turn an estimate for $\sum_g f(g)^2$ into an estimate for $\sum_g f(g)^{k+1}$.  Since we already understand the $k=1$ case (it is essentially the same as the dot product problem discussed above), this reduction allows us to obtain a non-trivial result.  Our first theorem is as follows.

\begin{thm}
\label{MT1}
Let $q$ be a power of an odd prime, and let $E\subset\mathbb{F}_q^2$ satisfy $|E|\gtrsim q^s$, where $s={2-\frac{1}{k+1}}$.  Then $\mathcal{C}_{k+1}(E)\gtrsim \mathcal{C}_{k+1}(\mathbb{F}_q^2)$.
\end{thm}

In addition to proving this theorem, we will consider the case where the finite field $\mathbb{F}_q$ is replaced by the ring $\mathbb{Z}/p^\ell\mathbb{Z}$.  The structure of the proof is largely the same; the dot product problem over such rings is studied in \cite{CIP}, giving us the $k=1$ case, and the machinery which lifts that case to arbitrary $k$ works the same way.  However, many details in the proofs are considerably more complicated.  The theorem is as follows.

\begin{thm}
\label{MT2}
Let $p$ be an odd prime, let $\ell\geq 1$, and let $E\subset(\mathbb{Z}/p^\ell\mathbb{Z})^2$ satisfy $|E|\gtrsim \ell^{\frac{2}{k+1}}p^{\ell s}$, where $s=2-\frac{1}{\ell(k+1)}$.  Then $\mathcal{C}_{k+1}(E)\gtrsim \mathcal{C}_{k+1}((\mathbb{Z}/p^\ell\mathbb{Z})^2).$
\end{thm}

We first note that, as we would expect, Theorem \ref{MT2} coincides with Theorem \ref{MT1} in the case $\ell=1$.  We also note that, for fixed $p$ and $k$, the exponent in Theorem \ref{MT2} is always less than 2, but it tends to $2$ as $\ell\to\infty$.  This does not happen in the finite field case, where the exponent depends on $k$ but not on the size of the field.  \\

Finally, we want to state the extent to which these results are sharp.  There are examples which show that the exponent must tend to $2$ as $k\to\infty$ in the finite field case, and as either $k\to\infty$ or $\ell\to\infty$ in the $\mathbb{Z}/p^\ell\mathbb{Z}$ case.

\begin{thm}[Sharpness]
We have the following:

\begin{enumerate}[i]
	\item For any $s<2-\frac{2}{k+1}$, there exists $E\subset\mathbb{F}_q^2$ such that $|E|\approx q^s$ and $\mathcal{C}_{k+1}(E)=o(\mathcal{C}_{k+1}(\mathbb{F}_q^2))$.
	\item For any $s<2-\min\left(\frac{2}{k+1},\frac{1}{\ell}\right)$, there exists $E\subset(\mathbb{Z}/p^\ell\mathbb{Z})^2$ such that $|E|\approx p^{\ell s}$ and $\mathcal{C}_{k+1}(E)=o(\mathcal{C}_{k+1}((\mathbb{Z}/p^\ell\mathbb{Z})^2))$.
\end{enumerate}

\end{thm}

\section{Characterization of the equivalence relation in terms of the $\text{SL}_2(R)$ action}

Our main tool in reducing the problem of $(k+1)$-point configurations to the $k=1$ case is the fact that we can express the equivalence relation in terms of the action of the special linear group; with some exceptions, tuples $x$ and $y$ are equivalent if and only if there exists a unique $g\in \text{SL}_2$ such that for each $i$, we have $y^i=gx^i$.  In order to use this, we need to bound the number of exceptions to this rule.  This is easy in the finite field case, and a little more tricky in the $\mathbb{Z}/p^\ell\mathbb{Z}$ case.  The goal of this section is to describe and and bound the number of exceptional configurations in each case.  We begin with a definition.

\begin{dfn}
Let $R$ be a ring.  A configuration $x=(x^1,...,x^{k+1})\in (R^2)^{k+1}$ is called \textbf{good} if there exist two indices $i,j$ such that $x^i\cdot x^{j\perp}$ is a unit.  A configuration is \textbf{bad} if it is not good.
\end{dfn}

As we will see, the good configurations are precisely those for which equivalence is determined by the action of $\text{SL}_2(R)$.  To prove this, we will need the following theorems about determinants of matrices over rings, which can be found in \cite{DF}, section 11.4.

\begin{thm}
\label{DF1}
Let $R$ be a ring, let $A_1,...,A_n$ be the columns of an $n\times n$ matrix $A$ with entries in $R$.  Fix an index $i$, and let $A'$ be the matrix obtained from $A$ by replacing column $A_i$ by $c_1A_1+\cdots+c_nA_n$, for some $c_1,...,c_n\in R$.  Then $\det(A')=c_i\det(A)$.
\end{thm}

\begin{thm}
\label{DF2}
Let $R$ be a ring, and let $A$ be an $n\times n$ matrix with entries in $R$.  The matrix $A$ is invertible if and only if $\det(A)$ is a unit in $R$.
\end{thm}

\begin{thm}
\label{DF3}
Let $R$ be a ring, and let $A$ and $B$ be $n\times n$ matrices with entries in $R$.  Then $\det(AB)=\det(A)\det(B)$.
\end{thm}

We are now ready to prove that equivalence of good configurations is given by the action of the special linear group.

\begin{lem}
\label{SL2}
Let $R$ be a ring, and let $x,y$ be good configurations such that $x^i\cdot x^{j\perp}=y^i\cdot y^{j\perp}$ for every pair of indices $i,j$.  Then there exists a unique $g\in \text{SL}_2(R)$ such that $y^i=gx^i$ for each $i$.
\end{lem}

\begin{proof}
Because $x$ and $y$ are good, there exist indices $i$ and $j$ such that $x^i\cdot x^{j\perp}$ is a unit; equivalently, the determinant of the $2\times 2$ matrix with columns $x^i$ and $x^j$ is a unit.  Denote this matrix by $(x^i\ x^j)$.  By theorem \ref{DF2}, this matrix is invertible.  Let

\[
g=(y^i\ y^j)(x^i\ x^j)^{-1}.
\]

Since $g(x^i\ x^j)=(gx^i\ gx^j)$, it follows that $y^i=gx^i$ and $y^j=gx^j$.  Also note that by Theorem \ref{DF3}, we have $\det(g)=1$.  Let $n$ be any other index.  We want to write $x^n=ax^i+bx^j$; this amounts to solving the matrix equation

\[
\begin{pmatrix}
x_1^i & x_1^j \\
x_2^i & x_2^j
\end{pmatrix}
\begin{pmatrix}
a \\
b
\end{pmatrix}
=x^n
\]

Since we have already established the matrix $(x^i\ x^j)$ is invertible, we can solve for $a$ and $b$.  Similarly, let $y^n=a'y^i+b'y^j$.  By Theorem \ref{DF1}, we have $\det(x^i\ x^n)=b\det(x^i\ x^j)$ and $\det(y^i\ y^n)=b'\det(y^i\ y^j)$.  It follows that $b=b'$, and an analogous argument yields $a=a'$.  Therefore, 

\[
gx^n=g(ax^i+bx^j)=agx^i+bgx^j=ay^i+by^j=y^n.
\]

So, we have established existance.  To prove uniqueness, note that $g$ must satisfy $g(x^i\ x^j)=(y^i\ y^j)$, and since $(x^i\ x^j)$ is invertible we can solve for $g$.
\end{proof}

Now that we know that good tuples allow us to use the machinery we need, we must prove that the bad tuples are negligible.

\begin{lem}
\label{countbad}
Let $R$ be a ring and let $E\subset R^2$.  We have the following:
\begin{enumerate}[i]
	\item If $R=\mathbb{F}_q$, then $E^{k+1}$ contains $\lesssim q^k|E|$ bad tuples.  In particular, if $|E|\gtrsim q^{1+\e}$ for any constant $\e>0$, the number of bad tuples in $E^{k+1}$ is $o(|E|^{k+1})$.
	\item If $R=\mathbb{Z}/p^\ell \mathbb{Z}$, the number of bad tuples in $R^{k+1}$ is $\lesssim p^{(2\ell-1)(k+1)+1}$.  In particular, if $|E|\gtrsim p^{2\ell-1+\frac{1}{k+1}+\e}$ for any constant $\e>0$, then the number of bad tuples in $E^{k+1}$ is $o(|E|^{k+1})$.
\end{enumerate}
\end{lem}

\begin{proof}
We first prove the first claim.  Since the only non-unit of $\mathbb{F}_q$ is 0, a bad tuple must consist of $k+1$ points which all lie on a line through the origin.  Therefore, we may choose $x^1$ to be anything in $E$, after which the next $k$ points must be chosen from the $q$ points on the line through the origin and $x^1$.  \\

To prove the second claim, first observe that the number of tuples where at least one coordinate is a non-unit is $p^{2(\ell-1)(k+1)}$, which is less then the claimed bound.  So, it suffices to bound the set of bad tuples where all coordinates are units.  Let $B$ be this set.  Define

\[
\psi(x_1^1,x_2^1,\cdots ,x_1^{k+1},x_2^{k+1})=(p^{\ell-1}x_1^1,x_2^1,\cdots ,p^{\ell-1}x_1^{k+1},x_2^{k+1}).
\]

If $x\in B$, then $x^i\cdot x^{j\perp}$ is a non-unit, meaning it is divisible by $p$, and 

\[
(p^{\ell-1}x_1^i,x_2^i)\cdot (p^{\ell-1}x_1^j,x_2^j)=p^{\ell-1}x^i\cdot x^{j\perp}=0.
\]

Therefore, $\psi$ maps bad tuples $x$ to tuples $y$ with $y^i\cdot y^{j\perp}=0$, or $y_1^iy_2^j-y_1^jy_2^i=0$.  Rearranging, using the fact that the second coordinate of each $y^i$ is a unit, we conclude that $\frac{y_1^i}{y_2^i}$ is a constant independent of $i$ which is divisible by $p^{\ell-1}$.  In other words, each $y^i$ is on a common line through the origin and a point $(n,1)$ where $p^{\ell-1}|n$.  There are $p$ such lines, and once we fix a line there are $p^{\ell(k+1)}$ choices of tuples $y$.  Therefore, $|\psi(B)|\leq p\cdot p^{\ell(k+1)}$.  Finally, we observe that the map $\psi$ is $p^{(\ell-1)(k+1)}$-to-1.  This gives us the claimed bound on $|B|$.
\end{proof}

\begin{lem}
\label{flemma}
Let $R$ be either $\mathbb{F}_q$ or $\mathbb{Z}/p^\ell\mathbb{Z}$.  Let $E\subset R^2$, and let $G\subset E^{k+1}$ be the set of good tuples.  Suppose $|E|\gtrsim q^{1+\e}$ if $R=\mathbb{F}_q$ and $|E|\gtrsim p^{2\ell-1+\frac{1}{k+1}+\e}$ if $R=\mathbb{Z}/p^\ell\mathbb{Z}$.  For $g\in \text{SL}_2(R)$, define $f(g)=\sum_x E(x)E(gx)$.  Then

\[
|E|^{2(k+1)}\lesssim \mathcal{C}_{k+1}(E)\sum_{g\in\text{SL}_2(R)} f(g)^{k+1}.
\]
\end{lem}

\begin{proof}
By Cauchy-Schwarz, we have

\[
|G|^2\leq |G/\sim|\cdot |\{(x,y)\in G\times G:x\sim y\}|.
\]

By assumption and Lemma \ref{countbad}, $|E|^{k+1}\approx |G|$, and therefore the left hand side above is $\approx |E|^{2(k+1)}$.  Since $G\subset E^{k+1}$ the right hand side above is $\leq \mathcal{C}_{k+1}(E)|\{(x,y)\in G\times G:x\sim y\}|$.  It remains to prove $|\{(x,y)\in G\times G: x\sim y\}|\leq \sum_{g\in\text{SL}_2(R)} f(g)^{k+1}$.  By lemma \ref{SL2},

\[
|\{(x,y)\in G\times G: x\sim y\}|=\sum_{x,y\in G}\sum_{\substack{g \\ y=gx}} 1.
\]

By extending the sum over $G$ to one over all of $E^{k+1}$, we bound the above sum by

\begin{align*}
&\sum_{x,y\in E^{k+1}}\sum_{\substack{g \\ y=gx}} 1 \\
=&\sum_x E(x^1)\cdots E(x^{k+1})\sum_g E(gx^1)\cdots E(gx^{k+1}) \\
=&\sum_g\left(\sum_{x^1} E(x^1)E(gx^1)\right)^{k+1} \\
=&\sum_g f(g)^{k+1}
\end{align*}

\end{proof}

\section{Lifting $L^2$ estimates to $L^{k+1}$ estimates}

In both the case $R=\mathbb{F}_q$ and $R=\mathbb{Z}/p^\ell\mathbb{Z}$, results are known for pairs of points, which is essentially the $k=1$ case.  The finite field version was studied in \cite{HI}, and the ring of integers modulo $p^\ell$ was studied in \cite{CIP}.  In section 2, we defined a function $f$ on $\text{SL}_2(R)$ and related the number of equivalence classes determined by a set to the sum $\sum_g f(g)^{k+1}$.  Since results are known for the $k=1$ case, we have information about the sum $\sum_g f(g)^2$.  We wish to turn that into a bound for $\sum_g f(g)^{k+1}$.  This is achieved with the following lemma.

\begin{lem}
\label{induction}
Let $S$ be a finite set, and let $F:S\to \mathbb{R}_{\geq 0}$.  Let 

\[
A=\frac{1}{|S|}\sum_{x\in S}F(x)
\]

denote the average value of $F$, and 

\[
M=\sup_{x\in S}F(x)
\]

denote the maximum.  Finally, suppose

\[
\sum_{x\in S}F(x)^2=A^2|S|+R.
\]

Then there exist constants $c_k$, depending only on $k$, such that

\[
\sum_{x\in S}F(x)^{k+1}\leq c_k(M^{k-1}R+A^{k+1}|S|).
\]
\end{lem}

\begin{proof}
We proceed by induction.  For the base case, let $c_1=1$ and observe that the claimed bound is the one we assumed for $\sum_x F(x)^2$.  Now, let $\{c_k\}$ be any sequence such that $k\binom{k}{j}c_j\leq c_k$ holds for all $j<k$; for example, $c_k=2^{k^2}$ works.  Now, suppose the claimed bound holds for all $1\leq j<k$, and also observe that the bound is trivial for $j=0$.  By direct computation, we have

\begin{align*}
&\sum_{x\in S}(F(x)-A)^2 \\
=&\sum_{x\in S}F(x)^2-2A\sum_{x\in S}F(x)+A^2|S| \\
=&\sum_{x\in S}F(x)^2-A^2|S| \\
=&R.
\end{align*}

We also have

\[
\sum_{x\in S}F(x)^{k+1}=\sum_{x\in S}(F(x)-A)^kF(x)+\sum_{j=0}^{k-1}\binom{k}{j}(-1)^{k-j+1}A^{k-j}\sum_{x\in S}F(x)^{j+1}.
\]

To bound the first term, we simply use the trivial bound.  Since $F(x)\leq M$ for all $x$, $A\leq M$, and $F(x),A\geq 0$, we conclude $|F(x)-A|\leq M$ for each $x$.  Therefore,

\[
\sum_{x\in S}(F(x)-A)^kF(x)\leq M^{k-1}\sum_{x\in S}(F(x)-A)^2=M^{k-1}R.
\]

To bound the second term, we use the inductive hypothesis and the triangle inequality.  We have

\begin{align*}
&\left|\sum_{j=0}^{k-1}\binom{k}{j}(-1)^{k-j+1}A^{k-j}\sum_{x\in S}F(x)^{j+1}\right| \\
\leq &k\cdot\sup_{0\leq j<k} \binom{k}{j}A^{k-j}\sum_{x\in S}F(x)^{j+1} \\
\leq &k\cdot\sup_{0\leq j<k} \binom{k}{j}A^{k-j}c_j(M^{j-1}R+A^{j+1}|S|) \\
\leq & c_k \cdot\sup_{0\leq j<k}(A^{k-j}M^{j-1}R+A^{k+1}|S|)
\end{align*}

Since $A\leq M$, it follows that $A^{k-j}M^{j-1}R\leq M^{k-1}R$ for any $j<k$, so the claimed bound holds.
\end{proof}

\section{Some lemmas about the action of $\text{SL}_2(R)$}

\begin{lem}
\label{action}
Let $G$ be a finite group acting transitively on a finite set $X$.  Define $\ph:X\times X\to\mathbb{N}$ by $\ph(x,y)=|\{g\in G:gx=y\}|$.  We have

\[
\ph(x,y)=\frac{|G|}{|X|}
\]

for every pair $x,y$.  If $h:X\to\mathbb{C}$ and $x_0\in X$, then

\[
\sum_{g\in G} h(gx_0)=\frac{|G|}{|X|}\sum_{x\in X}h(x).
\]

\end{lem}

\begin{proof}
The second statement follows from the first by a simple change of variables.  To prove the first, we have

\[
\sum_{x,y\in X}\ph(x,y)=\sum_{g\in G}\sum_{\substack{x,y\in X \\ gx=y}}1.
\]

On the right, for any fixed $g$, one can choose any $x$ and there is a unique corresponding $y$, so the inner sum is $|X|$ and the right hand side is therefore $|G||X|$.  On the other hand, $\ph$ is constant.  To prove this, let $x,y,z,w\in X$ and let $h_1,h_2\in G$ such that $h_1x=z$ and $h_2w=y$.  This means for any $g$ with $gz=w$, we have $(g_2gh_1)x=y$, so $\ph(z,w)\leq \ph(x,y)$.  By symmetry, equality holds.  If $c$ is the constant value of $\ph(x,y)$, the left hand side above must be $c|X|^2$, and therefore $c=\frac{|G|}{|X|}$ as claimed.
\end{proof}

\begin{lem}
\label{SL2size}
We have $|\text{SL}_2(\mathbb{F}_q)|=q^3-q$ and $|\text{SL}_2(\mathbb{Z}/p^\ell\mathbb{Z})|=p^{3\ell}-p^{3\ell-2}$.
\end{lem}

\begin{proof}
We are counting solutions to the equation $ad-bc=1$ where $a,b,c,d\in\mathbb{F}_q$.  We consider two cases.  If $a$ is zero, then $d$ can be anything, and we must have $bc=1$.  This means $b$ can be anything non-zero, and $c$ is determined.  So, there are $q^2-q$ solutions with $a=0$.  With $a\neq 0$, $b$ and $c$ can be anything, and $d$ is determined, giving $q^3-q^2$ solutions in this case.  So, there are $(q^3-q^2)+(q^2-q)$ total solutions. \\

Next, we want to count solutions to $ad-bc=1$ with $a,b,c,d\in\mathbb{Z}/p^\ell\mathbb{Z}$.  The arguments are essentially the same as in the proof of the finite field case, but slightly more complicated because there are non-zero elements which are still not units.  We again consider separately two cases according to whether $a$ is a unit or not.  If $a$ is a unit, then $b,c$ can be anything and then $d$ is determined, so there are $(p^{\ell}-p^{\ell-1})p^{2\ell}$ such solutions.  If $a$ is not a unit, then $b$ and $c$ must be units, as otherwise $1$ would be divisible by $p$.  So there are $p^{\ell-1}$ choices for $a$, $p^{\ell}$ choices for $d$, $p^{\ell}-p^{\ell-1}$ for $b$, and $c$ is determined.  Putting this together, we get the claimed number of solutions.
\end{proof}

\section{Proof of Theorem \ref{MT1}}
We are now ready to prove theorem \ref{MT1}.

\begin{proof}
First observe that good tuples are equivalent to $\approx q^{3}$ distinct tuples, so there are $\approx q^{2k-1}$ equivalence classes of good tuples.  Since the only non-unit in the finite field case is 0, the bad tuples are all in the same equivalence class.  So, our goal is to prove $\mathcal{C}_{k+1}(E)\gtrsim q^{2k-1}$.  We first must prove the estimate

\[
\sum_g f(g)^2=\frac{|E|^4}{q}+O(q^2|E|^2).
\]

We expand the sum on the left hand side and change variables to obtain

\[
\sum_g f(g)^2=\sum_{x^1,x^2,y^1,y^2}E(x^1)E(x^2)E(y^1)E(y^2)\left(\sum_{\substack{g \\ gx=y}}1\right).
\]

We first observe we may ignore the pairs $x,y$ which are on a line through the origin.  This is because if $x^2=tx^1$ and $y^2=sy^1$, there will exist $g$ with $gx=y$ if and only if $t=s$, in which case there are $\approx q$ choices for $g$.  So, we have $|E|$ choices for $x^1$ and $y^1$, $q$ choices for $t$, and $\approx q$ choices for $g$ giving an error of $O(q^2|E|^2)$, as claimed.  For all other pairs $x,y$, the inner sum in $g$ is 1 if $x\sim y$ and 0 otherwise.  Therefore, if $\nu(t)=|\{(x,y)\in E\times E:x\cdot y^\perp=t\}|$, we have

\[
\sum_g f(g)^2=O(|E|^2q^2)+\sum_t\sum_{\substack{x^1, x^2, y^1, y^2 \\ x^1\cdot x^{2\perp}=t \\ y^1\cdot y^{2\perp}=t}}E(x^1)E(x^2)E(y^1)E(y^2)=O(|E|^2q^2) +\sum_t\nu(t)^2.
\]

The proof of theorem 1.4 in \cite{HI} shows that $\nu(t)=\frac{|E|^2}{q}+O(|E|q^{1/2})$, so this gives

\[
\sum_t \nu(t)^2-\frac{|E|^4}{q}=\sum_t \left(\nu(t)-\frac{|E|^2}{q}\right)^2=O(|E|^2q^2),
\]

which proves the equation above.  We now apply lemma \ref{induction} with $F=f$.  Lemmas \ref{SL2size} and \ref{action} imply 

\[
A=\frac{1}{|\text{SL}_2(\mathbb{F}_q)|}\sum_xE(x)\sum_gE(gx)=\frac{1}{(q^2-1)}|E|^2=\frac{|E|^2}{q^2}+O\left(\frac{|E|^2}{q^4}\right)
\]

and

\[
|S|=q^3+O(q^2).
\]

Putting this together gives

\[
A^2|S|=\frac{|E|^4}{q}+O\left(\frac{|E|^4}{q^2}\right),
\]

and therefore

\[
\sum_g f(g)^2=A^2|S|+R
\]

with $R=O(q^2|E|^2)$.  Finally, we observe that $f$ has maximum $M\leq |E|$.  Therefore, lemma \ref{induction} gives

\[
\sum_g f(g)^{k+1}\lesssim q^2|E|^{k+1}+\frac{|E|^{2(k+1)}}{q^{2k-1}}.
\]

Together with lemma \ref{flemma}, this gives

\[
|E|^{2(k+1)}\lesssim \mathcal{C}_{k+1}(E)\left(q^2|E|^{k+1}+\frac{|E|^{2(k+1)}}{q^{2k-1}}\right).
\]

If the second term on the right is bigger, we get the result for free.  If the first term is bigger, we get

\[
\mathcal{C}_{k+1}(E)\gtrsim \frac{|E|^{k+1}}{q^2}.
\]

This will be $\gtrsim q^{2k-1}$ when $|E|\gtrsim q^{2-\frac{1}{k+1}}$, as claimed.

\end{proof}

\section{Size of $\mathcal{C}_{k+1}((\mathbb{Z}/p^\ell\mathbb{Z})^2)$}

Since $|\text{SL}_2(R)|\approx |R|^3$, we expect each tuple in $(R^2)^{k+1}$ to be equivalent to $\approx |R|^3$ other tuples, and therefore we expect the number of congruence classes to be $|R|^{2k-1}$.  In the finite field case, this was proved as the first step of the proof of Theorem \ref{MT1}, but the proof in the $R=\mathbb{Z}/p^\ell\mathbb{Z}$ is more complicated so we will prove it here, separately from the proof of Theorem \ref{MT2} in the next section.

\begin{thm}
\label{target}
We have $\mathcal{C}_{k+1}((\mathbb{Z}/p^\ell\mathbb{Z})^2)\approx (p^\ell)^{2k-1}$.  More precisely, the good $(k+1)$-point configurations of $(\mathbb{Z}/p^\ell\mathbb{Z})^2$ determine $\approx (p^\ell)^{2k-1}$ classes, and the bad configurations determine $o((p^\ell)^{2k-1})$ classes.
\end{thm}

\begin{proof}
We first establish that there are $\approx p^{\ell(2k-1)}$ classes of good tuples.  This is easy; if $x$ is a good tuple, we have seen the map $g\mapsto gx$ is injective, so each class has size $\approx p^{3\ell}$ and there are $p^{2\ell (k+1)}$ tuples, meaning there are $p^{2\ell(k+1)-3\ell}$ classes.  \\

It remains to bound the number of bad classes.  We first establish the $k=1,2$ cases.  When $k=1$, we want to prove there are $o(p^\ell)$ equivalence classes.  This is clear, because in the $k=1$ case we are looking at pairs $(x^1,x^2)$ whose class is determined by the scalar $x^1\cdot x^{2\perp}$.  The classes therefore correspond to the underlying set of scalars in $\mathbb{Z}/p^\ell\mathbb{Z}$, and the bad classes correspond to non-units.  In the $k=2$ case, we are looking at triples $(x^1,x^2,x^3)$ whose class is determined by the three scalars $(x^1\cdot x^{2\perp},x^2\cdot x^{3\perp},x^3\cdot x^{1\perp})$.  So, the space of equivalence classes can be identified with $((\mathbb{Z}/p^\ell\mathbb{Z})^2)^3$, and the bad classes correspond to triples of non-units.  \\

For $k\geq 3$, we use the following theorem, which is really just a more specific version of Theorem \ref{DF2}, also found in \cite{DF}, chapter 11.  \\

\begin{thm*}[\ref{DF2}']
For any $2\times 2$ matrix $A$, there exists a $2\times 2$ matrix $B$ with $AB=BA=(\det(A))I_2$, where $I_2$ is the $2\times 2$ identity matrix.
\end{thm*}

We also make a more specific version of the definition of good and bad tuples.  Namely, let$x$ be a $(k+1)$ point configuration in $\mathbb{Z}/p^\ell\mathbb{Z}$, and let $m\leq \ell$ be minimal with respect to the property that $p^m$ divides $x^i\cdot x^{j\perp}$ for every pair of indices $(i,j)$.  We say that $x$ is \textbf{$m$-bad}.  Observe that according to our previous definition, good tuples are $0$-bad and bad tuples are $m$-bad for some $m>0$.  Also observe that $m$-badness is preserved by equivalence, so we may define $m$-bad equivalence classes analogously.  An easy variant of the argument in Lemma \ref{countbad} shows that the number of $m$-bad tuples is $\lesssim p^{(2\ell-m)(k+1)+m}$; note that this bound can be rewritten as $p^{\ell(2k-1)+3\ell-km}$.  We claim that every $m$-bad equivalence class has at least $p^{3\ell-2m}$ elements.  It follows from the claim that there are $\lesssim p^{\ell(2k-1)+(2-k)m}$ $m$-bad classes, and since we may assume $k\geq 3$ the theorem follows from here.  To prove the claim, note that the equivalence class containing $x$ also contains $gx$ for any $g\in\text{SL}_2(\mathbb{Z}/p^\ell\mathbb{Z})$, so for a lower bound on the size of a class we need to determine the size of the image of the map $g\mapsto gx$.  First note that we may assume without loss of generality that each coordinate of $x^1$ is a unit.  This is because given $x$ we can shift any factor of $p$ from $x^1$ onto each other vector $x^i$ and obtain another representative of the same equivalence class.  Next, observe that if $x$ is $m$-bad and $gx=hx$, then by Theorem \ref{DF2}' we have $p^mg=p^mh$.  It follows that $h=g+p^{\ell-m}A$ for some matrix $A$ with entries between $0$ and $p^m$.  Using the fact that

\[
\det(A+B)=\det(A)+\det(B)+\mathcal{B}(A,B),
\]

where $\mathcal{B}$ is bilinear, we conclude that if $h=g+p^{\ell-m}A$ and $\det(g)=\det(h)=1$, we must have

\[
0=p^{2(\ell-m)}\det(A)+p^{\ell-m}\mathcal{B}(g,A).
\]

Let $m'$ be the minimal power of $p$ which divides all entries of $A$.  Since the entries of $g$ cannot all be divisible by $p$, it follows that $\ell-m+m'$ is the maximal power of $p$ which divides the second term above.  Since $2(\ell-m+m')$ divides the first term, it follows that both terms must be 0 for the equation to hold.  In particular, we must have $\mathcal{B}(g,A)=0$.  Since at least one entry of $g$ must be a unit, we can solve for one entry of $A$ in terms of the others.  Now observe that in order to have $gx=hx$, we must have $p^{\ell-m}Ax=0$.  In particular, $p^{\ell-m}Ax^1=0$.  Since each coordinate of $x^1$ is a unit, we may solve for another entry of the matrix $A$.  This means there are at most $p^{2m}$ choices for $A$, and hence the map $g\mapsto gx$ is at most $p^{2m}$-to-one.  It follows that $m$-bad classes have at least $p^{3\ell-2m}$ elements, as claimed.

\end{proof}

\section{Proof of Theorem \ref{MT2}}

\begin{proof}
In keeping with the rest of this paper, the proof of the $\mathbb{Z}/p^\ell\mathbb{Z}$ case is essentially the same as the finite field case, but more complicated casework is required to deal with non-units.  By our work in the previous section, our goal is to show $\mathcal{C}_{k+1}(E)\gtrsim (p^\ell)^{2k-1}$.  Following the line of reasoning in the proof of Theorem \ref{MT1}, we want to establish the estimate

\[
\sum_g f(g)^2=\frac{|E|^4}{p^\ell}+O(\ell^2|E|^2(p^\ell)^{3-\frac{1}{\ell}}).
\]

We have, after a change of variables,

\[
\tag{$*$}
\sum_g f(g)^2=\sum_{x^1,x^2,y^1}E(x^1)E(x^2)E(y^1)\sum_{\substack{g \\ gx^1=y^1}}E(gx^2).
\]

We first want to throw away terms where $x^1,y^1$ have non-units in their first coordinates.  Note that there are $\approx p^{4\ell-2}$ such pairs.  For each, there are $|E|$ many choices for $x^2$.  We claim that there are $\leq p^\ell$ choices of $g$ which map $x^1$ to $y^1$ under this constraint.  It follows from this claim that those terms contribute $\lesssim p^{5\ell-2}|E|$ to ($*$), which is less then the claimed error term.  To prove the claim, observe that we are counting solutions to the system of equations

\begin{align*}
ax_1^1+bx_2^1&=y_1^1 \\
cx_1^1+dx_1^1&=y_2^1\\
ad-bc&=1
\end{align*}

in $a,b,c,d$.  Since $x_1^1$ is a unit, we can solve the first two equations for $a$ and $c$, respectively.  Plugging these solutions into the third equation yields

\[
1=\frac{y_1^1}{x_1^1}d-\frac{y_2^1}{y_1^1}b.
\]

Since $y_1^1$ is a unit, for every $b$ there is a unique $d$ satisfying the equation.  This proves the claim.  Now, we want to remove all remaining terms from ($*$) corresponding to $x^1,x^2$ where $x^1\cdot x^{2\perp}$ is not a unit.  To bound this contribution, we observe that for any such pair, we can write $x^2=tx^1+k$, where $0<t<p$ and $k$ is a vector where both entries are non-units.  Therefore, there are $\leq |E|^2$ choices for $(x^1,y^1)$, there are $\leq p^{2\ell-1}$ choices for $x^2$, and there are $\leq p^\ell$ choices for $g$ as before.  This gives the bound $|E|^2p^{3\ell-1}$, smaller than the claimed error term.  This means, up to the error term, ($*$) can be written as

\[
\sum_{\substack{x^1,x^2,y^1,y^2 \\ x^1\cdot x^{2\perp}=y^1\cdot y^{2\perp}}}E(x^1)E(x^2)E(y^1)E(y^2)=\sum_t \nu(t)^2,
\]

where $\nu(t)=|\{(x,y)\in E\times E:x\cdot y^\perp=t\}|$.  This function was studied in \cite{CIP}; in that paper, it is proved that $\nu(t)=\frac{|E|^2}{q}+O(\ell|E|(p^\ell)^{\frac{1}{2}(2-\frac{1}{\ell})})$, leading to the claimed estimate for $\sum_g f(g)^2$, using the same reasoning as in the proof of Theorem \ref{MT1}.  Applying Lemma \ref{induction} and Lemma \ref{flemma} with $A\approx \frac{|E|^2}{p^{2\ell}},|S|\approx p^{3\ell},M\leq |E|, R=O(\ell^2 |E|^2(p^\ell)^{3-\frac{1}{\ell}})$ gives

\[
|E|^{2(k+1)}\lesssim \mathcal{C}_{k+1}(E)\left(\ell^2|E|^{k+1}(p^\ell)^{3-\frac{1}{\ell}}+\frac{|E|^{2(k+1)}}{p^{\ell(2k-1)}}\right).
\]

If the second term on the right is bigger, we get the result for free.  If the first term is bigger, we have

\[
\mathcal{C}_{k+1}(E)\gtrsim \frac{|E|^{k+1}}{\ell^2p^{3\ell-1}}.
\]

If $|E|\gtrsim \ell^{\frac{2}{k+1}}p^{\ell s}$, then this is $\gtrsim p^{\ell s(k+1)-3\ell+1}$, which is $\gtrsim p^{\ell(2k-1)}$ when $s\geq 2-\frac{1}{\ell(k+1)}$.

\end{proof}

\section{Proof of sharpness}

\begin{proof}
We first consider the finite field case.  Let $1\leq s<2-\frac{2}{k+1}$, and let $E$ be a union of $q^{s-1}$ circles of distinct radii.  Since each circle has size $\approx q$, this is a set of size $\approx q^s$.  Observe that for any $x\in E^{k+1}$ and any $g$ in the orthogonal group $O_2(\mathbb{F}_q)$, we have $gx\in E^{k+1}$.  Therefore, every configuration of points in $E$ is equivalent to at least $|O_2(\mathbb{F}_q)|\approx q$ other configurations.  This means that 

\[
\mathcal{C}_{k+1}(E)\lesssim q^{-1}|E|^{k+1}\approx q^{s(k+1)-1}=o(q^{2k-1}),
\]

where in the last step we use the assumed bound on $s$.  \\

Now, consider the $\mathbb{Z}/p^\ell\mathbb{Z}$ case.  Let $1\leq s<2-\min\left(\frac{2}{k+1},\frac{1}{\ell}\right)$.  We consider two different examples, according to which of $\frac{2}{k+1}$ or $\frac{1}{\ell}$ is smaller.  In the first case, the example that works for finite fields also works here; circles still have size $\approx p^\ell$, so nothing is changed.  In the second, let

\[
E=\{(t+pn,t+pm):0\leq t<p,0\leq m,n\leq p^{\ell-1}\}.
\]

Clearly $|E|=p^{2\ell -1}=(p^\ell)^{2-\frac{1}{\ell}}$, but it is also easy to check that $x\cdot y^\perp$ is never a unit for any $x,y\in E$.  Therefore, every configuration of points in $E$ is bad, and we have shown that this is $o(\mathcal{C}_{k+1}(p^{\ell(2k-1)}))$.

\end{proof}

\end{document}